\documentclass[12pt]{article}

\usepackage{amsmath}
\allowdisplaybreaks[4]
\usepackage{tikz}
\usetikzlibrary{decorations.pathreplacing, matrix}
\usetikzlibrary{decorations.pathreplacing}
\usepackage[all]{xy}
\usepackage{amssymb}
\usepackage{amsthm}
\usepackage{hyperref}
\usepackage{bm}
\hypersetup{colorlinks=true,linkcolor=blue,citecolor=red}
\usepackage{amsmath}
\usepackage{amscd}
\usepackage{verbatim}
\usepackage{eurosym}
\usepackage{float}
\usepackage{color}
\usepackage{dcolumn}
\usepackage[mathscr]{eucal}
\usepackage[all]{xy}
\usepackage{hyperref}
\usepackage{mathrsfs}
\usepackage{amsmath}
\usepackage{amssymb}
\usepackage{amsfonts,ifpdf}
\usepackage{graphicx}
\usepackage{times}
\usepackage{float}
\usepackage{epstopdf}
\usepackage{cite}
\usepackage{youngtab}
\usepackage{ytableau}
\ytableausetup
{mathmode, boxsize=0.9em}
\usepackage{indentfirst}

\newtheorem{theorem}{Theorem}[section]
\newtheorem{lemma}[theorem]{Lemma}
\newtheorem{prop}[theorem]{Proposition}

\theoremstyle{definition}

\newtheorem{pf}{proof}[section]
\theoremstyle{remark}

\makeatletter \@addtoreset{equation}{section} \makeatother
\makeindex 

\def\A{\mathcal{A}}

\def\B{\mathcal{B}}
\renewcommand\ell{l}

\def\min{\mbox{\rm min}}
\def\max{\mbox{\rm max}}

\def\And{\mbox{\rm ~and~}}

\begin{document}
	
	\begin{center}
	{\Large\bf
		Level of Faces for Exponential Sequence of Arrangements
	}\\ [7pt]
\end{center}
\vskip 3mm
\begin{center}
	Yanru Chen$^{1}$, Houshan Fu$^{2}$, Weikang Liang$^{3,*}$, Suijie Wang$^{4}$\\[8pt]
	$^{1,3,4}$School of Mathematics\\
	Hunan University\\
	Changsha 410082, Hunan, P. R. China\\[12pt]
	$^{2}$School of Mathematics and Information Science\\
	Guangzhou University\\
	Guangzhou 510006, Guangdong, P. R. China\\[12pt]
	
	 Emails: $^{1}$yanruchen@hnu.edu.cn, $^{2}$fuhoushan@gzhu.edu.cn,\\ ~~~~~~~~~
	 $^{3}$kangkang@hnu.edu.cn,
	 $^{4}$wangsuijie@hnu.edu.cn\\[15pt]
\end{center}
\vskip 3mm

\begin{abstract}
In this paper, we introduce the bivariate exponential generating function $F_l(x,y)$ for the number of  level-$l$ faces of an exponential sequence of arrangements (ESA),  and establish the  formula $F_l(x,y)=\big(F_1(x,y)\big)^l$ with a combinatorial interpretation. Its specialization at $x=0$ recovers a result first obtained by Chen {\em et al.} \cite{CWYZ2024,CFWY2025} for certain classic ESAs and later generalized to all ESAs by Southerland {\em et al.}  \cite{SSZ2025}.  As a byproduct,  we obtain that an alternating sum of the number of  level-$l$ faces is invariant with respect to the choice of ESA, and is exactly the Stirling number of the second kind.  We also extend the binomial-basis expansion theorem \cite{CWYZ2024,CFWY2025,Zhang2025}  and Stanley's formula on ESAs \cite{Stanley1996} from characteristic polynomials to Whitney polynomials.

\vskip 6pt
\noindent{\bf Keywords:} Whitney polynomial, Stirling number of the second kind, Hyperplane arrangement, exponential sequence of arrangements, \vspace{1ex}\\
{\bf Mathematics Subject Classifications:} 52C35, 05A10, 05A19
\end{abstract}

\section{Introduction}\label{Sec-1}
Let $\A$ be a hyperplane arrangement in the Euclidean space $\mathbb{R}^n$. Its \emph{intersection poset} $L(\A)$ is the set of all nonempty intersections of some hyperplanes in $\A$, ordered by reverse inclusion. A more general setting and background about hyperplane arrangements can be found in \cite{Stanley2007}. The \emph{Whitney polynomial} of $\A$ is a bivariate polynomial defined by
\begin{equation}\label{Eq1}
w(\A, x, t):=\sum_{X,Y\in L(\A),\, X\le Y}\mu(X,Y)x^{n-\dim{X}}t^{\dim{Y}},
\end{equation}
where $\mu$ is the M\"{o}bius function defined recursively by $\mu(X,X)=1$ and $\mu(X,Y)=-\sum_{X\le Z<Y}\mu(X,Z)$ for $X<Y$.  When $x=0$, it reduces to the \emph{characteristic polynomial} $\chi(\A,t):=w(\A,0,t)$. Conversely, the Whitney polynomial admits an expansion in terms of the characteristic polynomials as follows
\begin{equation}\label{Whitney2}
w(\A, x,t)=\sum_{X\in L(\A)}x^{n-\dim{X}}\chi(\A/X,t),
\end{equation}
where $\A/X:=\{H\cap X\ne\emptyset \colon H\in \A,  X \nsubseteq H\}$ is the \emph{restriction} of $\mathcal{A}$ to $X$. Each connected component of the \emph{complement} $M(\mathcal{A}/X):= X-\bigcup_{X\nsubseteq H,\, H\in\mathcal{A}}H\cap X$ is called a \emph{face} of $\A$. In particular, an $n$-dimensional face  of $\A$ is called a \emph{region}. 

First studied by Athanasiadis \cite{Athan1998} and Postnikov and  Stanley \cite{P-S2000}, a \emph{deformed braid arrangement} $\A_n$ is an arrangement in $\mathbb{R}^n$ of the form
\begin{equation}\label{deformed braid}
	\A_n \colon x_{i}-x_{j}=a_{ij}^{(1)},\ldots,a_{ij}^{(m_{ij})},\quad 1\le i<j\le n,
\end{equation}
where $m_{ij}\ge 1$ and $a_{ij}^{(1)}<\cdots<a_{ij}^{(m_{ij})}$.  Introduced by Stanley \cite{Stanley1996}, an \emph{exponential sequence $(\A_1,\A_2,\ldots)$ of arrangements} (abbreviated as ESA) consists of deformed braid arrangements $\A_n$ such that for any nonempty subset $S\subseteq [n]:=\{1,2,\ldots,n\}$.
\[L(\A_n^S)\cong L(\A_{|S|}) ,\] 
where $\A_n^S=\{H\in \A_n\colon H\text{ is parallel to }x_i-x_j=0\text{ for some }i,j\in S\}$.  The most classic ESAs include the braid arrangement, extended Shi/Catalan/semiorder arrangement, and Linial arrangement. 

Defined by Ehrenborg \cite{Ehrenborg2019},  the \emph{level} $l(A)$ of a subset $A \subseteq \mathbb{R}^n$ is the smallest dimension of all the subspaces $W$ of $\mathbb{R}^n$ such that the distances $d(\bm x,W)$ for all $\bm x\in A$ are uniformly bounded. Equivalently, there is a positive number $r$,
\begin{equation*}
A\subseteq B(W,r):=\big\{\bm x \colon d(\bm x,W)\le r\big\}.
\end{equation*}
Indeed, if the subset is a face $F$ of a hyperplane arrangement, $l(F)$  coincides with the degree of freedom of $F$ as defined by Armstrong and Rhoades in \cite{A-R2012}, and $l(F)-1$ is exactly the ideal dimension of $F$ introduced by Zaslavsky\cite{Zaslavsky2003}.  In this paper, we work with Ehrenborg's definition of the level.   For a hyperplane arrangement $\A$, 
let  $\mathcal{F}_{d,l}(\A)$ denote the set of all faces of  $\A$ with dimension $d$ and level $l$, and let $f_{d,l}(\A)=|\mathcal{F}_{d,l}(\A)|$. Obviously, $f_{d,l}(\A)=0$ when $d<l$.  For the deformed braid arrangement $\A_n$, $f_{d,0}(\A_n) = 0$ and $f_{d,1}(\A_n)$ is the number of relatively bounded $d$-dimensional faces of $\A_n$.

Associated with an ESA $(\A_1,\A_2,\ldots)$, for each level $l$, define the bivariate exponential generating function as
\[
F_{l}(x,y):=\sum_{n \ge d \ge l}f_{d,l}(\A_n)x^{n-d}\frac{y^n}{n!}.
\]
Recently, the study of $f_{d,l}$ and $F_{l}(x,y)$ for ESAs has drawn considerable attention \cite{Ehrenborg2019,CWYZ2024,CFWY2025,Zhang2025,SSZ2025}. When $\A_n$ is the extended Shi arrangement,  Ehrenborg \cite{Ehrenborg2019}  derived a counting formula for $f_{d,l}(\A_n)$ using difference operations. When $\A_n$ is the Catalan/semiorder-type arrangement, Chen {\em et al.} \cite{CWYZ2024} established an elegant formula for the exponential generating functions $F_{l}(0,y)$, i.e.,
\begin{equation}
	F_{l}(0,y)=\big(F_{1}(0,y)\big)^l.  \label{level-eq1}
\end{equation}
Additionally, Chen {\em et al.} \cite{CWYZ2024} provided an explicit formula for  $f_{d,l}(\A_n)$ when $\A_n$ is the extended Catalan arrangement. Thanks to Hetyei's labeling \cite{Hetyei2024},  Chen {\em et al.} \cite{CFWY2025}  extended the formula \eqref{level-eq1} to a broader class of ESAs. Applying Zaslavsky's convolution formula for level counting, Southerland {\em et al.} \cite{SSZ2025} proved that the number of regions of hyperplane arrangements with level $l$ is determined by its combinatorics and generalized the formula \eqref{level-eq1}  to all ESAs.  In this paper, this formula will be further extended to the bivariate generating function $F_{l}(x,y)$.  
\begin{theorem}\label{thm1}
Let $(\A_1,\A_2,\ldots)$ be an ESA and $l$ a positive integer. Then
	\[
	F_{l}(x,y)=\big(F_{1}(x,y)\big)^l.
	\]
\end{theorem}
Applying Theorem \ref{thm1}, we obtain in Theorem \ref{stirling} that an alternating sum of the number of  level-$l$ faces is invariant with respect to the choice of ESA, and it coincides with the Stirling number $S(n,l)$ of the second kind. A combinatorial interpretation of this result would be of great interest.
\begin{theorem}\label{stirling}
Let $(\A_1,\A_2,\ldots)$ be an ESA. Then
\[
\sum_{n\ge d\ge l}(-1)^{d-l}f_{d,l}(\A_n)=l!S(n,l).
\]
\end{theorem}
Let $f_l(\A_n)=f_{n,l}(\A_n)$ denote the number of regions of $\A_n$ with level $l$. Chen {\em et al.} \cite{CFWY2025} observed an interesting phenomenon: when $\A_n$ is the Catalan/semiorder-type arrangement, the characteristic polynomial $\chi(\A_n,t)$ has $f_l(\A_n)$ as the coefficient in its expansion in the polynomial basis  $\binom{t}{l}$. Soon after, Zhang \cite{Zhang2025} extended this result to all deformed braid arrangements  by induction. In this paper, we further  extend Zhang's result from the characteristic polynomial to the Whitney polynomial. An analogous result concerning the Whitney polynomial of non-degenerate deformations of the Coxeter arrangement of type $B$ will be stated in Theorem \ref{typeB}.
\begin{theorem}\label{thm2}
Let $\A_n$ be a deformed braid arrangement. Then
	\[
	w(\A_n,x,t)=\sum_{n \ge d \ge l \ge 1}(-1)^{d-l}f_{d,l}(\A_{n})x^{n-d}\binom{t}{l}.
	\]
\end{theorem}
For any ESA $(\A_1,\A_2,\ldots)$, Stanley \cite{Stanley1996} established a nice relation between the characteristic polynomials  and the regions of $\A_n$, as follows:
\[
1+\sum_{n\ge 1}\chi(\mathcal{A}_n,t)\frac{y^n}{n!}=\Big(1+\sum_{n\ge 1}(-1)^nr(\mathcal{A}_n)\frac{y^n}{n!}\Big)^{-t},
\]
where $r(\mathcal{A}_n)$ denotes the number of regions of $\mathcal{A}_n$. In this paper, it will be extended to a relation between the Whitney polynomials and the faces of $\A_n$. Particularly, taking $x=0$ in Theorem \ref{thm3}, we can recover Stanley's result.
\begin{theorem}\label{thm3} 
	Let $(\A_1,\A_2,\ldots)$ be an ESA and $F(x,y)=\sum_{l\ge 1}F_{l}(x,y)$. Then
	\[
	1+ \sum_{n\ge 1}w(\A_n,x,t) \frac{y^n}{n!}=\big(1-F_{1}(-x,-y)\big)^{t}=\big(1+F(-x,-y)\big)^{-t}.
	\]
\end{theorem}

The remainder of the paper is organized as follows. Section \ref{Sec2} is devoted to proving the main results, Theorem \ref{thm1}-\ref{thm3}. 
In Section \ref{Sec4}, we provide a combinatorial interpretation on Theorem \ref{thm1} by establishing an explicit bijection.
	
\section{Proofs of  main results}\label{Sec2}
Let $\A=\{H_i:\bm\alpha_i\cdot\bm x=a_i\mid i=1,\ldots,m\}$ be a hyperplane arrangement in $\mathbb{R}^n$, and $W(\mathcal{A})$ denote the vector space spanned by all the normal vectors $\bm\alpha_i$ of hyperplanes in $\mathcal{A}$. A face $ F$ of $\mathcal{A}$ is \emph{relatively bounded} if $F \cap W(\mathcal{A})$ is bounded. Denote by $b_{d}(\A)$ the number of relatively bounded faces of $\mathcal{A}$ with dimension $d$. In particular, $b_n(\mathcal{A})$ is the number of its relatively bounded regions and $r(\A)$ denotes the total number of its regions. The celebrated Zaslavsky's formula \cite{Zaslavsky1975} states that
\begin{equation}\label{Zas}
b_n(\mathcal{A}) = |\chi(\mathcal{A},1)|, \quad\quad r(\mathcal{A}) = |\chi(\mathcal{A},-1)|.
\end{equation}
\subsection{Proof of Theorem \ref{thm1}}	\label{Sec2-1}
Note that for a deformed braid arrangement $\mathcal{A}_n$, $W(\mathcal{A}_n)=\{(x_1,\ldots,x_n)\in \mathbb{R}^n\colon x_1+\cdots+x_n=0 \}$ is a subspace of codimension $1$. Thus, a face of $\mathcal{A}_n$ is relatively bounded if and only if its level is $1$. Immediately, we have for $1\le d\le n$,
\begin{equation}
\label{eq3}
f_{d,1}(\mathcal{A}_n)=b_d(\mathcal{A}_n).
\end{equation}
	
	For any hyperplane $H\colon\bm\alpha\cdot \bm x=a$ in $\mathbb{R}^n$, the \emph{centralization} of $H$ is the corresponding linear hyperplane $\underline{H}\colon \bm\alpha\cdot \bm x=0$. Naturally, the \emph{centralization} $\underline{\A}$ of $\A$ is a linear arrangement in $\mathbb{R}^n$ consisting of centralizations of all hyperplanes in $\mathcal{A}$, i.e., $\underline{\A}=\{\underline{H} \colon H\in \A\}$. For a linear subspace $V\subseteq\mathbb{R}^n$, the \emph{localization} of $\A$ at $V$ is a subarrangement of $\mathcal{A}$ defined by	
	\[
	\A_V:=\{H\in\A\colon V\subseteq H \text{ or }V\cap H=\emptyset\}.
	\]
Zaslavsky \cite{Zaslavsky2003} further derived a  convolution formula for level counting. 

\begin{lemma}[\cite{Zaslavsky2003}, Theorem 3.1]\label{Zas-level}
Let $\mathcal{A}$ be a hyperplane arrangement in $\mathbb{R}^n$, and $0\le l\le d\le n$. Then
\[
f_{d,l}(\A)=\sum_{X\in L(\underline{\A}),\,\dim(X)=l}r(\underline{\A}/X)b_d(\A_X).
\]
\end{lemma}
	
Notice that for any deformed braid arrangement $\A_n$, its centralization $\underline{\A_n}$ coincides with the classical braid arrangement in $\mathbb{R}^n$, whose hyperplanes are given by: $x_i-x_j=0$ for $1\le i<j\le n$. Thus, every element $X \in L(\underline{\A_n})$ leads to a natural partition $P(X) = \{\pi_1, \ldots, \pi_l\}$ of the set $[n]$ such that $l = \dim(X)$, and $i,j\in [n]$ lie in the same block $\pi_k$ if and only if the hyperplane $H : x_i-x_j =0$ contains $X$. With such a partition, we can establish the next lemma.
	
	\begin{lemma}\label{bounded}
		Let $(\A_1,\A_2,\ldots)$ be an ESA. Then for any $X\in L(\underline{\A_n})$ with $P(X)=\{\pi_1,\ldots,\pi_l\}$ and $d\ge l$, we have
		\[
		b_d\big((\A_n)_X\big)=\sum_{\substack{d_1+\cdots+d_l=d,\\d_1,\ldots,d_l\ge 1}}\prod_{i=1}^lb_{d_i}(\A_{|\pi_i|}).
		\]
	\end{lemma}
	\begin{proof}
By Zaslavsky’s formula in \eqref{Zas} and the definition of faces, we immediately deduce that
\[
b_d\big((\A_n)_X\big) = \sum_{Y \in L((\mathcal{A}_n)_X),\, \dim(Y)=d} \left|\chi\big((\mathcal{A}_n)_X/Y, 1\big)\right|.
\]
Comparing this with the coefficients of the Whitney polynomial in \eqref{Whitney2}, we observe that $b_d\big((\A_n)_X\big)$ is precisely the absolute value of the coefficient of $x^{n-d}$ in $w\big((\mathcal{A}_n)_X, x, 1\big)$. Namely,
\begin{equation}\label{eq2}
b_d\big((\A_n)_X\big)=\left| [x^{n-d}]w\big((\mathcal{A}_n)_X, x, 1\big) \right|.
\end{equation}
		Note from \cite[Lemma 7.2]{SSZ2025} that $
		L\big((\mathcal{A}_n)_X\big) \cong L(\mathcal{A}_{|\pi_1|}) \times \cdots \times L(\mathcal{A}_{|\pi_l|}).$ Therefore, 
		\[
		w\big((\mathcal{A}_n)_X, x, t\big) = w(\mathcal{A}_{|\pi_1|}, x, t) \cdots w(\mathcal{A}_{|\pi_l|}, x, t).
		\]
		Substituting this into \eqref{eq2}, we get
		\[
		\begin{aligned}
			b_d\big((\A_n)_X\big) &= \left| [x^{n-d}]w\big((\mathcal{A}_n)_X, x, 1\big) \right| \\
			&= \sum_{\substack{d_1+\cdots+d_l=d, \\ d_1,\ldots,d_l\ge 1}} \prod_{i=1}^l \left| [x^{|\pi_i| - d_i}]w(\mathcal{A}_{|\pi_i|}, x, 1) \right| \\
			&= \sum_{\substack{d_1+\cdots+d_l=d, \\ d_1,\ldots,d_l\ge 1}} \prod_{i=1}^l b_{d_i}(\mathcal{A}_{|\pi_i|}).
		\end{aligned}
		\]
This completes the proof.
	\end{proof}
With the preceding lemmas established, we are now ready to prove Theorem \ref{thm1}.
\begin{proof}[Proof of Theorem \ref{thm1}]
Note that for any $X\in L(\underline{\A_n})$, the intersection poset of  $\underline{\A_n}/X$ is isomorphic to that of the braid arrangement in $\mathbb{R}^{\dim(X)}$. It follows from Zaslavsky's formula in \eqref{Zas} that $r(\underline{\A_n}/X) = \dim(X)!$. Let $\Pi_{n,l}$ denote the set of all partitions of $[n]$ with $l$ blocks. From \eqref{eq3}, Lemma \ref{Zas-level} and Lemma \ref{bounded}, we have
		\begin{align*}
			f_{d,l}(\A_n)&=\sum_{X\in L(\underline{\A_n}),\,\dim(X)=l}r(\underline{\A_n}/X)b_{d}\big((\A_n)_X\big)\\
			&=\sum_{\{\pi_1,\ldots,\pi_l\}\in\Pi_{n,l}}l! \sum_{\substack{d_1+\cdots+d_l=d,\\d_1,\ldots,d_l\ge 1}}\prod_{i=1}^{l}b_{d_i}(\A_{|\pi_i|})\\
			&=\sum_{\substack{n_1+\cdots+n_l=n,\\n_1,\ldots,n_l\ge1}}\binom{n}{n_1,\ldots,n_l}\sum_{\substack{d_1+\cdots+d_l=d,\\d_1,\ldots,d_l\ge 1}}\prod_{i=1}^{l}f_{d_i,1}(\A_{n_i})\\
			&=\sum_{\substack{n_1+\cdots+n_l=n,\\d_1+\cdots+d_l=d,\\n_{i}\ge d_{i}\ge1}}\binom{n}{n_1,\ldots,n_l}\prod_{i=1}^{l}f_{d_i,1}(\A_{n_i}).
		\end{align*}
	Substituting this into the expression of $F_{l}(x,y)$,  we obtain
	\begin{align*}
		F_{l}(x,y)&=\sum_{n\ge d\ge l}f_{d,l}(\A_n)x^{n-d}\frac{y^n}{n!}\\
		&=\sum_{n\ge d\ge l}\sum_{\substack{n_1+\cdots+n_l=n,\\d_1+\cdots+d_l=d,\\n_{i}\ge d_{i}\ge1}}\binom{n}{n_1,\ldots,n_l}\Big(\prod_{i=1}^{l}f_{d_i,1}(\A_{n_i})\Big)x^{n-d}\frac{y^n}{n!}\\
		&=\prod_{i=1}^{l}\Big(\sum_{n_{i}\ge d_{i}\ge 1}f_{d_i,1}(\A_{n_i})x^{n_i-d_i}\frac{y^{n_i}}{n_i!}\Big)\\
		&=\big(F_{1}(x,y)\big)^l.
	\end{align*}
This completes the proof.
	\end{proof}
	
\subsection{Proof of Theorem \ref{stirling}}
The \emph{Stirling number of the second kind} $S(n,l)$ is the number of partitions of the set $[n]$ into $l$ blocks, i.e., 
$S(n,l)=\frac{1}{l!}\sum\binom{n}{n_1,\ldots,n_l}$, where the sumation is over all integers $n_1,\ldots,n_l\ge 1$ with $n_{1}+\cdots+n_{l}=n$.  For any ESA $(\A_1,\A_2,\ldots)$,  the following combinatorial identity shows that an alternating sum of the numbers of level-$l$ faces in $\mathcal{A}_n$ is  independent of the specific choice of ESA, and is exactly equal to $l! S(n,l)$.
\begin{proof}[Proof of Theorem \ref{stirling}]
Note that for any face $F$ of $\mathcal{A}$, $F$ is a region of $\mathcal{A}_n/{\text{aff}(F)}$, where  the affine hull aff$(F)$ of $F$ is the maximal element of $L(\A)$ containing $F$. Applying Zaslavsky's formula \eqref{Zas}, we have
\begin{align*}
			\sum_{d=1}^{n}(-1)^{d-1}f_{d,1}(\A_{n})&=\sum_{d=1}^{n}\sum_{{X\in L(\A_n),\,\dim{X}=d}}(-1)^{d-1}b_d(\A_{n}/{X})\\
			&=\sum_{X\in L(\A_n)}\chi(\A_{n}/X,1).
\end{align*}
Together with the definition of the Whitney polynomial in \eqref{Eq1}, we further arrive at
\[
\sum_{d=1}^{n}(-1)^{d-1}f_{d,1}(\A_{n})=w(\A_{n},1,1).
\]
From its alternative expression in \eqref{Whitney2}, we obtain
		\[
		w(\A_{n},1,1)=\sum_{X\le Z}\mu(X,Z)=\sum_{Z\in L(\A_{n})}\sum_{X\leq Z}\mu(X,Z)=1.
		\]
It follows that $\sum_{d=1}^{n}(-1)^{d-1}f_{d,1}(\A_{n})=1$. Combining Theorem \ref{thm1}, we have
		\begin{align*}
			\sum_{n\ge l}\sum_{d=l}^{n}(-1)^{d}f_{d,l}(\A_{n})\frac{y^n}{n!}&=\Big(\sum_{n \ge 1}\sum_{d=1}^{n}(-1)^{d}f_{d,1}(\A_{n})\frac{y^n}{n!}\Big)^{l}\\
			&=\Big(\sum_{n\ge 1}-\frac{y^n}{n!}\Big)^{l}\\
			&=\sum_{n\ge l}\sum_{\substack{n_{1}+\cdots+n_{l}=n,\\n_1,\ldots,n_l\ge 1}}(-1)^{l}\frac{y^n}{n_{1}!\cdots n_{l}!}.
		\end{align*}
		By comparing the coefficients of $y^n$ on both sides of the above equality, we obtain
		\[
		\sum_{d=l}^{n}(-1)^{d}f_{d,l}(\A_{n})=(-1)^{l}\sum_{\substack{n_{1}+\cdots+n_{l}=n,\\n_1,\ldots,n_l\ge 1}}\binom{n}{n_1,\ldots,n_l}=(-1)^{l} l!S(n,l).
		\]
This completes the proof.
	\end{proof}	
	
\subsection{Proof of Theorem \ref{thm2}}\label{Sec2-2}
In this subsection, we first give a proof of  Theorem \ref{thm2}, and then obtain an analogous result regarding the Whitney polynomial of non-degenerate deformations of the Coxeter arrangement of type $B$. To this end, we introduce a closely related result on the characteristic polynomials of deformed braid arrangements due to Zhang's work. The first main result in \cite{Zhang2025} showed that the characteristic polynomial of a deformed braid arrangement can be expanded in terms of the number of regions with different levels.
\begin{lemma}[\cite{Zhang2025}, Theorem 1.1]\label{Zhang1}
Let $\A_{n}$ be a deformed braid arrangement in $\mathbb{R}^n$. Then
\[
\chi(\A_{n}, t) = \sum_{l=1}^{n} (-1)^{n-l} f_{n,l}(\A_{n}) \binom{t}{l}.
\]
\end{lemma}
In fact, the formula in Lemma \ref{Zhang1} can be extended to general restriction $\mathcal{A}_n/X$ of an arbitrary deformed braid arrangement $\mathcal{A}_n$ to an element $X\in L(\A_n)$.
More precisely, Zhang demonstrated in \cite[Lemma 3.1]{Zhang2025} that the restriction $\mathcal{A}_n/H$ of  $\mathcal{A}_n$ to a hyperplane $H\in \mathcal{A}_n$ is still a deformed braid arrangement. Thus, the restriction $\mathcal{A}_n/X$ of  $\mathcal{A}_n$ to an arbitrary element $X\in L(\A_n)$ is also a deformed braid arrangement. Immediately, the subsequent result follows from Lemma \ref{Zhang1}.
\begin{prop}\label{restriction}
Let $\A_{n}$ be a deformed braid arrangement in $\mathbb{R}^n$. Then for any $X \in L(\A_{n})$ with dimension $d$, we have
\[
\chi(\A_{n}/{X}, t) = \sum_{l=1}^{d} (-1)^{d-l} f_{d,l}(\A_{n}/{X}) \binom{t}{l}.
\]
\end{prop}
	
Applying Proposition \ref{restriction}, we give a simple proof of Theorem \ref{thm2}.
\begin{proof}[Proof of Theorem \ref{thm2}]
Note the fact that 
\[
f_{d,l}(\A_{n})=\sum_{X\in L(\A_n),\,\dim(X)=d}f_{d,l}(\A_{n}/{X}).
\]
It follows from \eqref{Whitney2} and Proposition \ref{restriction} that 
		\begin{align*}
			w(\A_{n},x,t)&=\sum_{X\in L(\A_{n})}x^{n-\dim{X}}\chi(\A_{n}/{X},t)\label{Whitney1} \\
			&=\sum_{d=1}^{n}x^{n-d}\sum_{X\in L(\A_{n}),\,\dim{X}=d}\sum_{l=1}^{d}(-1)^{d-l}f_{d,l}(\A_{n}/{X})\binom{t}{l}\\
			& =\sum_{n \ge d \ge l \ge 1}(-1)^{d-l}f_{d,l}(\A_{n})x^{n-d}\binom{t}{l}.
		\end{align*}
This completes the proof.
\end{proof}
	
Zhang \cite{Zhang2025} further considered a similar expansion for the characteristic polynomials of deformations of  the Coxeter arrangement of type $B$. A \emph{non-degenerate deformation} $\mathcal{B}_n$ of the Coxeter arrangement of  type $B_{n-1}$ in $\mathbb{R}^n$ is given by
\begin{align*}
\mathcal{B}_n&:=\{x_i - x_j= a_{ij}^{(1)},\ldots,a_{ij}^{(m_{ij})}\mid 1\le i<j\le n\}\\
&\cup\{x_i + x_j = b_{ij}^{(1)},\ldots,b_{ij}^{(s_{ij})}\mid1\le i<j\le n\}\\
&\cup \{x_i= c_{i}^{(1)},\,\ldots,\,c_{i}^{(t_i)}\mid 1\le i\le n\},
\end{align*}
where $m_{ij}, s_{ij}, t_i \geq 1$ for all $i,j$ with $1\leq i<j \leq n$.
\begin{lemma}[\cite{Zhang2025}, Theorem 1.4]\label{Zhang2}
Let $\B_{n}$ be a non-degenerate deformation of the Coxeter arrangement of type $B_{n-1}$ in $\mathbb{R}^n$.  Then
\[
\chi(\mathcal{B}_{n}, t) = \sum_{l=0}^{n} (-1)^{n-l} f_{n,l}(\mathcal{B}_{n}) \binom{\tfrac{t-1}{2}}{l}.
\]
\end{lemma}
Notably, the non-degenerate deformations of the Coxeter arrangement of type $B$ share the analogous property with the deformed braid arrangements. Namely, the restriction $\mathcal{B}_n/X$ of a non-degenerate deformation $\B_n$ of an arbitrary Coxeter arrangement of type $B_{n-1}$ to an element $X\in L(\B_n)$ is a non-degenerate deformation of the Coxeter arrangement of type $B$ as well (see \cite[Lemma 3.2]{Zhang2025}). The property guarantees the formula in Lemma \ref{Zhang2} can be generalized to such restrictions $\mathcal{B}_n/X$ of $\B_n$.
\begin{prop}\label{restriction2}
Let $\B_{n}$ be a non-degenerate deformation of the Coxeter arrangement of type $B_{n-1}$ in $\mathbb{R}^n$. Then for any $X \in L(\B_{n})$ with dimension $d$, we have
\[
\chi(\mathcal{B}_{n}/X, t) = \sum_{l=0}^{d} (-1)^{d-l} f_{d,l}(\mathcal{B}_{n}/X) \binom{\tfrac{t-1}{2}}{l}.
\]
\end{prop}
 
Based on Zhang's work, we further extend the formula to the Whitney polynomials of non-degenerate deformations of the Coxeter arrangement of type $B$. Using Proposition \ref{restriction2}, Theorem \ref{typeB} can be proved analogously to Theorem \ref{thm2}, and its proof is therefore omitted.
\begin{theorem}\label{typeB}
Let $\B_{n}$ be a non-degenerate deformation of the Coxeter arrangement of type $B_{n-1}$ in $\mathbb{R}^n$. Then 
\[
w(\mathcal{B}_n,x,t)=\sum_{n \ge d \ge l \ge 0}(-1)^{d-l}f_{d,l}(\mathcal{B}_n)x^{n-d}\binom{\frac{t-1}{2}}{l}.
\]
\end{theorem}
\subsection{Proof of Theorem \ref{thm3}}\label{Sec2-3}
Next we apply Theorems \ref{thm1} and \ref{thm2} to prove Theorem \ref{thm3}.
\begin{proof}[Proof of Theorem \ref{thm3}]
By Theorem \ref{thm1} and Theorem \ref{thm2}, we have
		\begin{align}\label{eq1}
			&1+ \sum_{n\ge 1}w(\A_{n},x,t)\frac{y^n}{n!}=1+\sum_{ n \geq d \geq l \geq 1}(-1)^{d-l}f_{d,l}(\A_{n})x^{n-d}\frac{y^n}{n!}\binom{t}{l}\notag\\
			&=1+\sum_{ n \geq d \geq l \geq 1}(-1)^{d-l}x^{n-d}\frac{y^n}{n!}\binom{t}{l}\sum_{\substack{n_{1}+\cdots+n_{l}=n,\\d_{1}+\cdots+d_{l}=d,\\n_{i}\ge d_{i}\ge 1}}\binom{n}{n_{1}\ldots n_{l}}\prod_{i=1}^{l}f_{d_{i},1}(\A_{n_i})\notag\\
			&=1+\sum_{n\ge d\ge l\ge 1}\binom{t}{l}\sum_{\substack{n_{1}+\cdots+n_{l}=n,\\d_{1}+\cdots+d_{l}=d,\\n_{i}\ge d_{i}\ge 1}}\prod_{i=1}^{l}\Big((-1)^{d_{i}-1}f_{d_{i},1}(\A_{n_i})x^{n_{i}-d_{i}}\frac{y^{n_i}}{n_{i}!}\Big)\notag\\
			&=\big(1-F_{1}(-x,-y)\big)^t.
		\end{align}
		On the other hand, applying Theorem \ref{thm1} again, we  deduce 
		\[
		\big(1-F_{1}(-x,-y)\big)^{-1}=1+\sum_{l\ge 1}\big(F_{1}(-x,-y)\big)^{l}=1+\sum_{l\ge 1}F_{l}(-x,-y).
		\] 
		Substituting this into \eqref{eq1}, we immediately obtain
		\[
		1+ \sum_{n\ge 1}w(\A_{n},x,t)\frac{y^n}{n!}=\big(1-F_{1}(-x,-y)\big)^t=\big(1+F(-x,-y)\big)^{-t}.
		\]
This completes the proof.
\end{proof}
	
\section{A bijection based on Hetyei’s labeling
}\label{Sec4}
Given an ESA $(\A_1,\A_2,\ldots)$, in this section we concentrate on constructing a bijection from faces of $\A_n$ with level $l$ to $(l+1)$-tuples $(\pi, F_1,\ldots, F_l)$, where $\pi=(\pi_1,\ldots,\pi_l)$ is an ordered $l$-partition of $[n]$, and $F_p$ is a relatively bounded face of $\A_{|\pi_p|}$ for $p=1,\ldots,l$. 

Let us review some notations on digraphs. Let $G=\big(V(G),E(G)\big)$ be a finite digraph with the vertex set $V(G)$ and the directed edge set $E(G)$. In the digraph $G$, two vertices $v$ and $u$ are said to be {\em strongly connected} if there exists a directed $(v,u)$-path from $v$ to $u$ and a directed $(u,v)$-path from $u$ to $v$. Alternatively,  two vertices $v$ and $u$ are strongly connected in $G$ if and only if there exists a directed cycle of $G$ containing $v$ and $u$. The digraph $G$ is called {\em strongly connected} if any two vertices $v$ and $u$ are strongly connected. Every maximal strongly connected subdigraph of $G$ is known as a {\em strong component} of $G$.

Recall from \eqref{deformed braid} that a deformed braid arrangement $\mathcal{A}_n$ in $\mathbb{R}^n$ consists of the following hyperplanes: 
\[
\A_n \colon x_{i}-x_{j}=a_{ij}^{(1)},\ldots,a_{ij}^{(m_{ij})}, \;\;1\le i<j\le n,
\]
where $m_{ij}\ge 1$ and $a_{ij}^{(1)} < \cdots < a_{ij}^{(m_{ij})}$ for any pair $1\le i < j\le n$. Most recently, Hetyei \cite{Hetyei2024} constructed a weighted digraph associated with each region of a deformed braid arrangement. Using the graph structure, Hetyei showed that all regions of such an arrangement can be bijectively labeled by a set of $m$-acyclic weighted digraphs, which contain only directed cycles with negative weights. This is a key idea for establishing our bijection.

Inspired by Hetyei's labeling approach, for any face $F$ of $\mathcal{\A}_n$, we construct a digraph on the vertex set $[n]$ as follows: take a point $\bm x=(x_{1},\ldots,x_{n})\in F$; for each $i<j$, if $x_{i}-x_{j}\geq a_{ij}^{(1)}$, we create a directed edge $i\rightarrow j$; if $x_{i}-x_{j} \leq a_{ij}^{(m_{ij})}$, we create a directed edge $i\leftarrow j$. Note that for each hyperplane $H\in\A_n$, all points of $F$ either lie on the hyperplane $H$ or lie in the same side of $H$ uniformly. Thus, this digraph is independent of the choice of representatives $\bm x\in F$ and denoted by $G(F)$. We shall use this digraph structure as a bridge to construct our desired bijection.

Let $\mathcal{A}_n$ be a deformed braid arrangement in $\mathbb{R}^n$ and $F$ is a face with $(x_1, \ldots, x_n) \in F$. Suppose $G_1(F),\ldots,G_l(F)$ are all strong components of $G(F)$. Note the following basic fact: by the construction of $G(F)$, vertices $i$ and $j$ from distinct strong components have exactly one directed edge between them (either $i\to j$ or $j\to i$). Now, consider two strong components $G_p(F)$ and $G_q(F)$. Suppose there exist  $i\in V\big(G_p(F)\big)$ and $j\in V\big(G_q(F)\big)$ with the directed edge $i\to j$. We can directly conclude that for any $u\in V\big(G_p(F)\big)$ and  $v\in V\big(G_q(F)\big)$, the unique directed edge between $u$ and $v$ is from $u$ to $v$. Otherwise, $V\big(G_p(F)\big)\sqcup V\big(G_q(F)\big)$ forms a strongly connected subdigraph of $G(F)$, which contradicts the assumption that $G_p(F)$ and $G_q(F)$ are  distinct strong components of $G(F)$. We now consider the simple digraph $\bar{G}(F)$ on strong components of $G(F)$, where each strong component is viewed as a vertex of  $\bar{G}(F)$ and a directed edge is given by $G_p(F) \rightarrow G_q(F)$ if there exists $i\in V(G_p(F)), j\in V(G_q(F))$ such that there is a directed edge from $i$ to $j$. Notice from \cite[exercise 1.4.14]{West} that for an arbitrary $n$-vertex digraph with no directed cycle, there exists a linear ordering $v_1,v_2,\ldots,v_n$ such that if $v_i\to v_j$, then $i<j$. Thus, the digraph $\bar{G}(F)$ leads to a linear ordering of the strong components of $G(F)$ in such a way that 
\begin{itemize}
	\item[{\rm $(\star)$}] for any $i$ in a preceding strong component and for any $j$ in a succeeding strong component there is a directed edge $i \to j$ but no directed edge $j\to i$.
\end{itemize}

Furthermore, suppose $\big(G_1(F),\ldots,G_l(F)\big)$ is the $l$-tuple of strong components of $G(F)$ induced by the linear ordering defined in $(\star)$. In this context, for any $i \in V\big(G_p(F)\big)$, $j \in V\big(G_q(F)\big)$ with $p < q$, the next relation directly follows from the construction of $G(F)$:
\begin{equation}\label{Inequ}
	\begin{cases}
		x_{i}-x_{j}>a_{ij}^{(m_{ij})}, & \text{if}\;\; i<j; \\ x_{j}-x_{i}<a_{ij}^{(1)}, & \text{if}\;\; i>j.
	\end{cases}
\end{equation}

Before proceeding further, the following lemmas are required.
\begin{lemma}\label{same-component}
	Let $\mathcal{A}_n$ be a deformed braid arrangement in $\mathbb{R}^n$ and $F$ a face of $\A_n$. For any $(x_1, \ldots, x_n) \in F$, if vertices $i,j$ of $G(F)$ are in the same strong component, then
	\[
	|x_{i}-x_{j}|\le (n-1)a,
	\]
	where $a=\max\big\{|a_{pq}^{(k)}|\mid 1\le p<q\le n, 1\le k\le m_{pq}\big\}$.
\end{lemma}
\begin{proof}
	Notice the simple fact that for an arbitrary directed edge $k\to s$ of $G(F)$, from the construction of $G(F)$, we have either $x_s - x_k \ge a_{sk}^{(1)}$ or $ x_s- x_k \ge -a_{sk}^{(m_{sk})}$.
	Since vertices $i,j$ of $G(F)$ are in the same strong component, there exists a directed path $P$ from $i$ to $j$. Suppose $P: i=i_0\to i_1\to\cdots\to i_l=j$. Then $l\leq n-1$, and 
	\[
	x_i - x_j \;=\; \sum_{c=1}^l (x_{i_{k-1}}-x_{i_k})\ge -la\ge -(n-1)a.
	\]
	Likewise, there exists a directed path from $j$ to $i$. Using the same argument in the previous case, we also obtain $x_j - x_i\ge -(n-1)a$. It follows that $|x_i-x_j|\le (n-1)a$. 
\end{proof}

Most recently, Chen et al.  showed in \cite[Theorem 2.7]{CFWY2025} that the level of any region of  a deformed braid arrangement is precisely the number of strong components of the associated $m$-acyclic weighted digraph. In fact, this property can be extended to the more general faces of deformed braid arrangements.
\begin{prop}\label{component-number}
	Let $\mathcal{A}_n$ be a deformed braid arrangement in $\mathbb{R}^n$ and $F$ a face of $\A_n$. Then the number of strong components of $G(F)$ equals the level $l(F)$ of $F$.
\end{prop}
\begin{proof}
	Suppose $\big(G_1(F),\ldots,G_l(F)\big)$ is the $l$-tuple of strong components of $G(F)$ induced by the linear ordering defined in $(\star)$. To prove $l(F)\le l$, we construct a subspace $W \subseteq \mathbb{R}^n$ as follows:
	\[
	W = \big\{(y_1,\dots,y_n)\in\mathbb{R}^n\mid  y_i = y_j \text{ if } i,j \text{ are in the same strong component of } G(F)\big\}.
	\]
	Trivially, the subspace $W$ has dimension $l$. Let $i_1 \in V\big(G_1(F)\big), \ldots, i_l \in V(G_l(F)\big)$ and $\bm x=(x_1,\ldots,x_n)\in F$. We take $\bm y=(y_1,\ldots,y_n)\in W$ satisfying that $y_{i}=x_{i_k}$ whenever $i\in V\big(G_k(F)\big)$. Then we have
	\[
	d(\bm x, W) \leq d(\bm x,\bm y)=\sqrt{\sum_{i=1}^{n}|x_{i}-y_{i}|^2}=\sqrt{\sum_{k=1}^{l}\sum_{i\in V(G_k(F))}|x_{i}-x_{i_k}|^2}.
	\]
	It follows from Lemma \ref{same-component} that $d(\bm x, W)$ is uniformly bounded for all $\bm x\in F$. Therefore, we get $l(F) \leq \dim(W) = l$.
	
	Conversely, to prove $l(F)\ge l$, we consider a cone $C \subseteq W$ defined by
	\[
	C=\big\{(y_1,\ldots,y_n)\in W\mid y_{i}\le y_{j}\text{ if }i\in V\big(G_p(F)\big), j\in V\big(G_q(F)\big) \text{ with }  p<q\big\}.
	\]
	Since $C$ is a cone, if $C \subseteq B(W', r)$ for some subspace $W'$ and positive number $r$, then $C \subseteq W'$. This implies that the cone $C$ has level $\dim(C) = l$. Take $\bm z=(z_1,\ldots,z_n)\in F$. According to \eqref{Inequ}, for any $\bm y = (y_1, \ldots, y_n) \in C$, and $i \in V\big(G_p(F)\big), j \in V\big(G_q(F)\big)$ with $p \leq q$, we have
	\[
	\begin{cases}
		(z_{i}+y_{i})-(z_{j}+y_{j})=z_{i}-z_{j}, &\text{ if } p=q;\\
		(z_{i}+y_{i})-(z_{j}+y_{j})\ge z_{i}-z_{j}>a_{ij}^{(m_{ij})},&\text{ if }p < q, i<j;\\
		(z_{i}+y_{i})-(z_{j}+y_{j})\le z_{j}-z_{i} < a_{ij}^{(1)},&\text{ if }p<q, i>j.
	\end{cases}
	\]
	So, for any hyperplane $H\in \A_{n}$, the points $\bm z$ and $\bm z+\bm y$ either both lie on $H$, or both lie in the same side of $H$. This indicates that $\bm z+\bm y\in F$, and hence
	\[
	\bm z+C=\{\bm z+\bm y\colon \bm y\in C\}\subseteq F.
	\]
	It follows that $C\subseteq B(F,|\bm z|)$. Note that \cite[Lemma 2.5]{CFWY2025} shows that for any sets $X$ and $Y$ in $\mathbb{R}^n$, if there exists $r>0$ such that $X\subseteq B(Y,r)$, then $l(X)\le l(Y)$. It follows that $l(F) \ge l(C) = l$. In conclusion, we arrive at $l(F)=l$.
\end{proof}

Note that for any $i,j \in [n]$, the poset isomorphism $L\big(\A_n^{\{i,j\}}\big)\cong L(\A_2)$ implies that the number of hyperplanes parallel to $x_i-x_j=0$ in $\A_n$ equals $|\A_2|$. Hence, each $\A_n$ has the following form
\[
\A_n\colon x_{i}-x_j=a_{ij}^{(1)},\ldots,a_{ij}^{(m)},\quad1\le i<j\le n,\; m\ge 1.
\]
Let $F$ be a face of $\A_n$ with level $l$, and $\bm{y} = (y_1, \ldots, y_n) \in F$. According to Proposition \ref{component-number}, we may assume $\big(G_1(F),\ldots,G_l(F)\big)$ is the $l$-tuple of all strong components of $G(F)$ induced by the linear ordering defined in $(\star)$. For each $p=1,\ldots,l$, let $V\big(G_p(F)\big)=\{s_1,\ldots,s_{n_p}\}$ with $s_1<\cdots<s_{n_p}$. By Theorem \ref{thm2}, the numbers $f_{d,l}(\A_{n_p})$ are determined by its intersection poset. Since $L(\A_{n_p}) \cong L(\A_n^{\{s_1, \ldots, s_{n_p}\}})$, we may identify $\A_{n_p}$ by
\begin{equation}\label{ESA}
x_i-x_j=a_{s_is_j}^{(1)},\ldots,a_{s_is_j}^{(m)},\quad 1\le i<j\le n_p,
\end{equation}
whose intersection poset is isomorphic to $L(\A_n^{\{s_1, \ldots, s_{n_p}\}})$. Consider a vector $\bm y^{(p)}$ in $\mathbb{R}^{n_p}$ by extracting the entries of $\bm y$ indexed by $V\big(G_p(F)\big)$, i.e., $\bm y^{(p)}=(y_{s_1},\ldots,y_{s_{n_p}})$. Then there exists a unique face $F_p$ of $\mathcal{A}_{n_p}$ such that $\bm y^{(p)} \in F_p$. It is clear that $G(F_p) \cong G_p(F)$, and hence $l(F_p) = 1$ by Proposition~\ref{component-number}. We now define a map 
\begin{equation}\label{bijection}
\phi: \mathcal{F}_{d,l}(\A_n) \longrightarrow\bigsqcup_{\substack{\pi = (\pi_{1},\ldots,\pi_{l})\in \mathrm{Par}_{l}[n], \\ d_{1}+\cdots+d_{l}=d}} \{\pi\} \times \mathcal{F}_{d_1,1}(\A_{|\pi_{1}|})\times\cdots\times \mathcal{F}_{d_l,1}(\A_{|\pi_{l}|})
\end{equation}
by $\phi(F) = (\pi, F_1, \ldots, F_l)$, where $\pi = \big(V\big(G_1(F)\big), \ldots, V\big(G_l(F)\big)\big)$ and $\mathrm{Par}_l[n]$ denotes the set of all ordered $l$-partitions of $[n]$. The following theorem states that $\phi$ is bijective.
\begin{theorem}\label{Bijection}
Let $(\A_1,\A_2,\ldots)$ be an ESA. Then the map $\phi$ defined as \eqref{bijection} is a bijection.
\end{theorem}
\begin{proof}
We first verify that the map $\phi$ is well-defined. The preceding argument shows that every $F_p$ is a level-$1$ face of the corresponding arrangement $\A_{|\pi_p|}$. Therefore, it remains only to prove that
\begin{equation}\label{Dim}
	\dim(F)=\dim(F_1)+\cdots+\dim(F_l).
\end{equation}
For any $1\le i<j\le n$, on the one hand,
since $F$ is a face of $\mathcal{A}_n$, we have
\[
y_i -y_j = a \in \{a_{ij}^{(1)},\ldots,a_{ij}^{(m)}\} \Longleftrightarrow F \subseteq H : x_i - x_j = a \in \mathcal{A}_n,
\]
in which case the digraph $G(F)$ contains both directed edges $i\to j$ and $j\to i$, that is, the vertices $i$ and $j$ belong to the same strong component $G_p(F)$ of $G(F)$.
On the other hand, by the definition of $\bm{y}^{(p)}$, we have
\[
y_i -y_j = a \in \{a_{ij}^{(1)},\ldots,a_{ij}^{(m)}\} \Longleftrightarrow F_p \subseteq H \colon x_{t_{i}} - x_{t_j} = a \in \mathcal{A}_{n_p},
\]
where $V\big(G_p(F)\big) = \{s_1, \ldots, s_{n_p}\}$ with $s_1<\cdots<s_{n_p}$. We assume $s_{t_i} = i, s_{t_j} = j$.
This yields an explicit one-to-one correspondence between the set $\{H\in\mathcal A_n\mid F\subseteq H\}$ and $\bigsqcup_{p=1}^l\{ H\in\mathcal A_{|\pi_p|}\mid F_p\subseteq H\}$ that sends each $H: x_i -x_j =a_{ij}^{(k)}$ in $\mathcal{A}_n$ containing $F$ to the hyperplane $x_{t_i}-x_{t_j}=a_{ij}^{(k)}$ in $\mathcal A_{n_p}$ containing $F_p$. In addition, we have
\[
\dim(F)=\dim\Big(\bigcap_{F\subseteq H,\, H\in\mathcal A_n} H\Big) \quad\And\quad \dim(F_p)=\dim\Big(\bigcap_{F_p\subseteq H,\, H\in\mathcal A_{|\pi_p|}} H\Big).
\]
Therefore, \eqref{Dim} holds, and hence $\phi$ is well-defined.

Next, we prove that $\phi$ is injective. Arguing by contradiction, suppose $F$ and $F^\prime$ are two faces of $\A_n$ such that $\phi(F)=\phi(F^\prime) = (\pi,F_1,\ldots,F_l)$. Let $(y_{1},\ldots,y_n)\in F$ and $ (y^{\prime}_1,\ldots,y^{\prime}_n)\in F^{\prime}$. Then $F$ and $F^\prime$ are separated by a hyperplane $H: x_i-x_j=a_{ij}^{(k)}$ in $\A_n$ for some $k=1,\ldots,m$. 
Thus, we may assume
\[
y_{i}-y_{j}<a_{ij}^{(k)}\quad\And\quad y^\prime_{i}-y^\prime_{j}\geq a_{ij}^{(k)}.
\]
If $i,j \in \pi_p$ for some $p$, we assume $\pi_p = \{s_1,\ldots,s_{n_p}\}$ with $s_1<\cdots<s_{n_p}$. Then $(y_{s_1}, \ldots, y_{s_{n_p}})$ and $(y_{s_1}^\prime, \ldots, y_{s_{n_p}}^\prime)$ belong to the same face $F_p$ of $\A_{n_p}$. It follows that $y_i-y_j<a_{ij}^{(k)}$ if and only if  $y^\prime_{i}-y^\prime_{j}<a_{ij}^{(k)}$, a contradiction. If $i \in \pi_p$, $j \in \pi_q$ for some $p \neq q$, then by \eqref{Inequ}, we have
	\[
	\begin{cases}
		y_{i}-y_{j}> a_{ij}^{(m)} \geq a_{ij}^{(k)} , & \text{if}\;\; p<q; \\ y_{j}^\prime-y_{j}^\prime <a_{ij}^{(1)} \leq a_{ij}^{(k)}, & \text{if}\;\; p>q,
	\end{cases}
	\]
	which also leads to a contradiction. Consequently, the map $\phi$ is an injection.
	
Finally, we show that $\phi$ is a surjection. Let $\pi = (\pi_1, \ldots, \pi_l) \in \mathrm{Par}_l[n]$ and $F_p$ be a face of $\A_{|\pi_{p}|}$ with $l(F_p) = 1$ for each $p=1,\ldots,l$. Take $y^{(p)} = (y^{(p)}_1, \ldots, y^{(p)}_{|\pi_p|}) \in F_p$. Note that $y^{(p)} + t(1,\ldots,1) \in F_p$ holds for any $t \in \mathbb{R}$, hence for any $p<q$, we may assume that 
	\begin{equation}\label{1.1a}
		\min_{i,j}\{y^{(p)}_i - y^{(q)}_j\} >\max\big\{|a_{12}^{(1)}|,\ldots,|a_{12}^{(m)}|,\ldots,|a_{(n-1)n}^{(1)}|,\ldots,|a_{(n-1)n}^{(m)}|\big\}.
	\end{equation}
	Let $\bm x$ be the $n$-dimensional vector obtained by concatenating $y^{(1)}, \ldots, y^{(l)}$ such that the components of $y^{(p)}$ are assigned sequentially to the positions indexed by $\pi_p$ for each $p \in \{1,\ldots,l\}$. Then there exists a unique face $F$ of $\mathcal{A}_n$ such that $\bm x \in F$.
	Let $G_p(F)$ be the subdigraph of $G(F)$ induced by the vertices set $\pi_p$, then $G_p(F) \cong G(F_p)$, and hence $G_p(F)$ is strong connected. Furthermore, for any $i \in \pi_p, j \in \pi_q$, if follows from \eqref{1.1a} that there is no directed path from $j$ to $i$ if $p <q$, and no directed path from $i$ to $j$ if $p > q$ in $G(F)$. As a result, $\big(G_1(F),\ldots,G_l(F)\big)$ is the tuple of  strong components of $G(F)$ induced by the linear ordering defined in $(\star)$. That is, $\phi(F) = (\pi, F_1, \ldots, F_l)$.
	
Summarising the preceding arguments, we have shown that $\phi$ is a bijection. 
\end{proof}
As a direct application of Theorem \ref{Bijection}, we give a bijective proof of  the identity $F_{l}(x,y)=\big(F_{1}(x,y)\big)^l$ in Theorem \ref{thm1}.	
\begin{proof}[Bijective proof of Theorem \ref{thm1}]
From the bijection $\phi$ in Theorem \ref{Bijection}, we have 
	\begin{align*}
		f_{d,l}(\A_n)=\sum_{\substack{n_{1}+\cdots+n_{l}=n,\\d_{1}+\cdots+d_{l}=d,\\n_{i}\ge d_{i}\ge 1}}\binom{n}{n_{1}\ldots n_{l}}\prod_{i=1}^{l}f_{d_{i},1}(\A_{n_{i}}).
	\end{align*}
It follows that
	\begin{align*}
		F_{l}(x,y)&=\sum_{n\ge d\ge l}f_{d,l}(\A_{n})x^{n-d}\frac{y^n}{n!}\\
		&=\prod_{i=1}^{l}\sum_{n_{i}\ge d_{i}\ge 1}f_{d_i,1}(\A_{n_i})x^{n_i -d_i}\frac{y^{n_i}}{n_{i}!}\\ &=\big(F_{1}(x,y)\big)^l.
	\end{align*}
	This completes the proof.
\end{proof}
\section*{Acknowledgements}
This paper is supported by National Natural Science Foundation of China (Grant No. 12571350) and Guangdong Basic and Applied Basic Research Foundation (Grant No. 2025A1515010457).


\begin{thebibliography}{99}\setlength{\itemsep}{-.0mm}
		\bibitem{A-R2012}
		D. Armstrong, B. Rhoades. The Shi arrangement and the Ish arrangement. Trans. Amer. Math. Soc. 364 (2012), 1509--1528.
		
		
		\bibitem{Athan1998}
		C. Athanasiadis. Deformations of Coxeter hyperplane arrangements and their characteristic polynomials. Arrangements-Tokyo 1998, 1--26, Adv. Stud. Pure Math. 27, Kinokuniya, Tokyo, 2000.
		
		
		
		\bibitem{CFWY2025}
		Y. Chen, H. Fu, S. Wang, J. Yang. Level of regions for deformed braid arrangements. J. Combin. Theory Ser. A 217 (2026), Paper No. 106077.
		
		\bibitem{CWYZ2024}
		Y. Chen, S. Wang, J. Yang, C. Zhao. Regions of level $\ell$ of Catalan/semiorder-type arrangements. arXiv:2410.10198v2 (2024).
		
		
		\bibitem{Ehrenborg2019}
		R. Ehrenborg. Counting faces in the extended Shi arrangement. Adv. in Appl. Math. 109 (2019), 55--64.
		
		
		
		\bibitem{Hetyei2024}
		G. Hetyei. Labeling regions in deformations of graphical arrangements. arXiv: 2312.06523v3 (2024).
		
		
		
		\bibitem{P-S2000}
		A. Postnikov, R. P. Stanley. Deformations of Coxeter hyperplane arrangements. In memory of Gian-Carlo Rota. J. Combin. Theory Ser. A  91 (2000), 544--597.
		
		
		\bibitem{SSZ2025}
		F. Southerland, L. Southern, S. Zhou. Region level via centralization for hyperplane arrangements and beyond. arXiv:2511.09653 (2025).
		
		
		\bibitem{Stanley1996}
		R. P. Stanley. Hyperplane arrangements, interval orders, and trees. Proc. Nat. Acad. Sci. U.S.A. 93 (1996), 2620--2625.

		
		\bibitem{Stanley2007}
		R. P. Stanley. An introduction to hyperplane arrangements. In Geometric  Combinatorics, 389--496, IAS/Park City Math. Ser. 13, Amer. Math. Soc. Providence, RI, 2007.
		
		
		\bibitem{West}
		D. B. West. Introduction to Graph Theory. Prentice Hall, Upper Saddle River, NJ, 1996.
		
		\bibitem{Zaslavsky1975}
		T. Zaslavsky. Facing up to arrangements: face-count formulas for partitions of space by hyperplanes. Mem. Amer. Math. Soc. 1 (1975), issue 1, no. 154, vii+102.
		
		\bibitem{Zaslavsky2003}
		T. Zaslavsky. Faces of a hyperplane arrangement enumerated by ideal dimension, with application to plane, plaids, and Shi. Geom. Dedicata 98 (2003), 63--80.
		
		\bibitem{Zhang2025}
		N. Zhang. Characteristic polynomials of deformations of Coxeter arrangements via levels of regions. S\'em. Lothar. Combin. 93B (2025), Art. 93B, 12 pp.
	\end{thebibliography}
\end{document}